\documentclass[12pt]{amsart}
\usepackage{amsmath}
\usepackage{amssymb}
\usepackage{amsthm}
\usepackage{amscd}
\usepackage{amsfonts}
\usepackage{graphicx}
\usepackage{fancyhdr}
\usepackage{epsfig}
\usepackage{epstopdf}
\usepackage{setspace}
\usepackage{verbatim}
\usepackage{subfigure}

\theoremstyle{plain} \numberwithin{equation}{section}
\newtheorem{theorem}{Theorem}[section]
\newtheorem{corollary}[theorem]{Corollary}

\newtheorem{lemma}[theorem]{Lemma}
\newtheorem{proposition}[theorem]{Proposition}

\theoremstyle{definition}
\newtheorem{definition}[theorem]{Definition}

\newtheorem{example}[theorem]{Example}

\DeclareMathOperator{\support}{supp } 
\DeclareMathOperator{\LCM}{LCM}
\DeclareMathOperator{\mi}{mi}
\DeclareMathOperator{\atoms}{atoms}
\DeclareMathOperator{\lcm}{lcm}
\DeclareMathOperator{\equivSet}{equiv}
  
\DeclareMathOperator{\Scarf}{scarf }

\title {Finite atomic lattices and resolutions of monomial ideals}

\author{Sonja Mapes}
\address{Department of Mathematics, Duke University,
117 Physics Bldg,
Box 90320
Durham, NC 27708 
}
\email{smapes@math.duke.edu}

\begin{document}
\begin{abstract}
In this paper we primarily study monomial ideals and their minimal free resolutions by studying their associated LCM lattices.  In particular, we formally define the notion of coordinatizing a finite atomic lattice $P$ to produce a monomial ideal whose LCM lattice is $P$, and we give a complete characterization of all such coordinatizations.  We prove that all relations in the lattice $\mathcal{L}(n)$ of all finite atomic lattices with $n$ ordered atoms can be realized as deformations of exponents of monomial ideals.  We also give structural results for $\mathcal{L}(n)$.  Moreover, we prove that the cellular structure of a minimal free resolution of a monomial ideal $M$ can be extended to minimal resolutions of certain monomial ideals whose LCM lattices are greater than that of $M$ in $\mathcal{L}(n)$.
\end{abstract}
\maketitle

\section{Introduction}
Let $M$ be a monomial ideal in a polynomial ring $R$.  We are interested in studying the minimal free resolution of $R/M$, and specifically understanding the maps in this resolution.  Our approach is heavily dependent on using the combinatorial structure of the lattice of least common multiples, or LCM lattice, associated to $M$, as well as the set of all such lattices for monomial ideals with a fixed number of generators.

The LCM lattice was introduced by Gasharov, Peeva, and Welker for the purpose of studying free resolutions of monomial ideals \cite{GPW}.  In particular they show that two monomial ideals with isomoprhic LCM lattices have isomorphic minimal resolutions.  This motivated Phan to ask, and answer in the affirmative, the question of whether or not every finite atomic lattice can be realized as the LCM lattice of some monomial ideal \cite{phan}.  Phan's work leads to the point of view that finite atomic lattices are ``abstract monomial ideals''.  

This paper concerns itself with two main types of results:  to give a complete characterization of how abstract monomial ideals are coordinatized;  and to use the set of all finite atomic lattices with a fixed number of atoms to understand families of resolutions of monomial ideals.  Theorem \ref{coordinatizations} gives a complete characterization of how to \emph{coordinatize} a finite atomic lattice by associating a monomial ideal $M_P$ to a finite atomic lattice $P$ such that the LCM lattice of $M_P$ is isomorphic to $P$.  Moreover, Proposition \ref{realizable} shows that all monomial ideals can be realized as coordinatizations of their LCM lattices.  

Theorem \ref{Ln} is a restatement of Phan's result which shows that the set $\mathcal{L}(n)$ of all finite atomic lattices with $n$ ordered atoms is itself a finite atomic lattice \cite{phan}, and this idea is central to the rest of our work.  Our interest in studying $\mathcal{L}(n)$ is due to the fact that in this setting, total Betti numbers are weakly increasing as one travels up chains in $\mathcal{L}(n)$ (see Theorem 3.3 in \cite{GPW}).  This is related to lower semi-continuity of Betti numbers under deforming the exponents of the monomial ideal.  A {\it deformation of the exponents} of a monomial ideal is a process by which one replaces a given monomial ideal $M$ with a ``deformed'' monomial ideal $M_{\epsilon}$, where one or more ties between exponents on minimal generators is broken.

Bayer, Peeva and Sturmfels show that the minimal resolution of $M_{\epsilon}$ provides a (not necessarily minimal) resolution of $M$ (see Theorem 4.3 in \cite{bps}).  The main result of Section~\ref{defoSection} is Theorem~\ref{chains=defos}, which states that every relation in $\mathcal{L}(n)$ can be realized as a deformation of exponents for some appropriate choice of coordinatization.

Ultimately we are interested in understanding questions such as:  how can we control deformations of the exponents so that Betti numbers do not increase; and what can we say about the cellular structure of minimal resolutions depending on the location in $\mathcal{L}(n)$?  In the remainder of this paper we focus on answering the latter problem as well as providing structural results that will be useful for considering either question.

Both Theorem 5.3 in \cite{velascoFrames}, and Proposition 5.15 in \cite{phan} describe simplicial complexes which support a minimal free resolution of a monomial ideal by virtue of being their Scarf complex.  Proposition \ref{generalizationScarf} generalizes this result by indicating a class of regular CW-complexes which supports the minimal free resolution of an ideal.  Additionally it is known that if $P > Q$ in $\mathcal{L}(n)$ then any minimal resolution of $P$ is a resolution of $Q$ (see Theorem 3.3 in \cite{GPW}).  If this resolution of $Q$ is minimal then it inherits any cellular structure which might have existed for the resolution of $P$.  Propositions \ref{simplicialFilter} and \ref{GenSimpFilt} show that to a certain extent the converse is true.  Namely that under certain hypotheses any cellular structure for a minimal resolution of $Q$ can be lifted to a minimal resolution of $P$.

Using this approach, of studying minimal resolutions of monomial ideals by examining the resolutions of ``nearby'' ideals in $\mathcal{L}(n)$, may tempt one to make a comparison with the upper-semicontinuity of Betti numbers along deformations of projective varieties in the Hilbert scheme.  However, a more apropos analogy may be with the MacPhersonian, which is a combinatorial version of the Grassmannian \cite{biss}, where monomial ideals play the role of vector spaces and finite atomic lattices replace oriented matroids.  To develop a theory of combinatorial moduli spaces modeled on both of these examples is a subject for future study.

Further results in this paper are Theorem \ref{gradedLn}, which states that $\mathcal{L}(n)$ is a graded lattice of rank $n$, and other structural results in Section \ref{structure} describing covering relations and meet-irreducibles in this lattice.  Additionally, Theorem \ref{genericCoords} gives a lattice theoretic description of strongly generic monomial ideals.      
        
\subsection*{Acknowledgements}
I would like to thank my advisor David Bayer for imparting his intuition and knowldge as I worked on this project in fulfillment of my Ph.D. requirements.  I would also like to give special recognizition to Jeff Phan, whose thesis introduced the ideas which are continued in this paper.  Finally I'd like to thank Ezra Miller for his support and interest in this project. 

\section{Preliminaries}\label{background}

A {\it lattice} is a set $(P, <)$  with an order relation $<$ which is transitive and antisymmetric satisfying the following properties:
\begin{enumerate}
\item $P$ has a maximum element denoted by $\hat{1}$
\item $P$ has a minimum element denoted by  $\hat{0}$
\item Every pair of elements $a$ and $b$ in $P$ has a join $a \vee b$ which is the least upper bound of the two elements
\item Every pair of elements $a$ and $b$ in $P$ has a meet $a \wedge b$ which is the greatest lower bound of the two elements.  
\end{enumerate}

If $P$ only satisfies conditions 2 and 4 then it is a {\it meet-semilattice}.  We will use the following (Proposition 3.3.1 in \cite{sta97}) several times.

\begin{proposition}\label{prop:Stanley}
Any meet-semilattice with a unique maximal element is a lattice.
\end{proposition}

We define an {\it atom} of a lattice $P$ to be an element $x \in P$ such that $x$ covers $\hat{0}$ (i.e. $x > \hat{0}$ and there is no element $a$ such that $x > a > \hat{0}$). We will denote the set of atoms as $\atoms (P)$.  
\begin{definition}
If every element in $P -\{\hat{0}\}$ is the join of atoms, then $P$ is an {\it atomic lattice}.  Furthermore, if $P$ is finite, then it is a {\it finite atomic lattice}. 
\end{definition}

If $P$ is a lattice, then we define elements $x \in P$ to be {\it meet-irreducible} if $x \neq a \wedge b$ for any $ a > x, b>x$. We denote the set of meet-irreducible elements in $P$ by $\mathrm{mi}(P)$. Given an element $x \in P$, the {\it order ideal} of $x$ is defined to be the set $\lfloor {x} \rfloor = \{a \in P | a \leqslant x\}$.  Similarly, we define the {\it filter} of $x$ to be $\lceil {x} \rceil = \{a \in P | x \leqslant a\}$. 

For the purposes of this paper it will often be convenient to consider finite atomic lattices as sets of sets in the following way.  Let $\mathcal{S}$ be a set of subsets of $\{1,...,n\}$ with no duplicates, closed under intersections, and containing the entire set, the empty set, and the sets $\{i\}$ for all $1 \leqslant i \leqslant n$.  Then it is easy to see $\mathcal{S}$ is a finite atomic lattice by ordering the sets in $\mathcal{S}$ by inclusion.  This set obviously has a minimal element, a maximal element, and $n$ atoms, so by Proposition~\ref{prop:Stanley} we need to show that it is a meet-semilattice.  Here the meet of two elements would be defined to be their intersection and since $S$ is closed under intersections this is a meet-semilattice.  Conversely, it is clear that all finite atomic lattices can be expressed in this way, simply by letting $$\mathcal{S}_P = \{ \sigma \,\vert\, \sigma = \support(p), p\in P\},$$ where $\support (p) = \{a_i \,\vert\, a_i \leqslant p, a_i \in \atoms(P)\}$.    

Lastly, there are two different simplicial complexes that one can associate to a finite atomic meet-semilattice $P$.  One is the order complex, $\Delta(P)$, which is the complex where the vertices are the elements of $P$ and the facets correspond to maximal chains of $P$.  Alternatively we can define a special case of the cross cut complex, which we will denote as $\Gamma(P)$, where the atoms correspond to vertices and simplices correspond to subsets of atoms which have a join or meet in $P$.  It is known that $\Delta(P)$ is homotopy equivalent to $ \Gamma(P)$ \cite{topMeth}.

\section{Coordinatizations}\label{coord}

Define a {\it labeling} of $P$ to be any assignment of non-trivial monomials $\mathcal {M} = \{m_{p_1}, ..., m_{p_t}\}$ to some set of elements $p_i \in P$.  It will be convenient to think of unlabeled elements as having the label $1$. Then a labeling is a {\it coordinatization} if the monomial ideal $M_{P, \mathcal{M}}$, which is generated by monomials \[x(a) = \prod_{ p \in \lceil{a}\rceil ^ c} m_p\] for each $a \in \atoms (P)$, has LCM lattice isomorphic to $P$. 

The following theorem and proof are a generalization of Theorem 5.1 in \cite{phan}.  

\begin{theorem}\label{coordinatizations}
Any labeling $\mathcal{M}$ of elements in a finite atomic lattice $P$ by monomials satisfying the following two conditions will yield a coordinatization of the lattice $P$.

\begin{itemize}
\item If $p \in \mi (P)$ then $m_p \not = 1$.  (i.e. all meet-irreducibles are labeled)
\item If $\gcd (m_p, m_q) \not = 1$ for some $p, q \in P$ then $p$ and $q$ must be comparable.  (i.e. each variable only appears in monomials along one chain in $P$.)
\end{itemize}
\end{theorem}

\begin{proof}
Let $P'$ be the LCM lattice of $M_{P, \mathcal{M}}$.  We just need to show that $P'$ is isomorphic to $P$.  For $b \in P$ define $f$ to be the map such that \[f(b) = \prod_{ p \in \lceil{b}\rceil ^ c} m_p.\]  We will show that this map is well-defined and that it is an isomorphism of lattices. In particular we must show that it is a lattice homomorphism (preserves joins and meets) as well as a bijection on sets. 

First note that obviously $f$ is a bijection on atoms and that \[\lceil {b} \rceil ^c = \bigcup_{a_i \in \support(b)} \lceil{a_i}\rceil ^c. \]  

In order to show that $f$ is well-defined we will show that 
\begin{equation*}
f(b) = \lcm \{f(a_i) | a_i \in \support(b) \}.
\end{equation*}
 This will also imply that $f$ is join-preserving and a surjection.  By the two remarks above, we know that \[f(b) = \prod m_p \] where $p \in \lceil{a_i}\rceil^c$ for at least one $a_i \in \support(b)$.  It is clear that $f(a_j)$ divides $f(b)$ for all $a_j\in\support(b)$. We then need to show that if $x_i^{m_i}$ is the highest power of $x_i$ dividing $f(b)$ then there is some $a_j\in\support(b)$ such that $x_i^{m_i}$ divides $f(a_j)$.
This follows from the fact that $x_i$ only divides monomials that label elements in one chain of $P$.  Indeed let $q$ be the largest element of this chain which is in $\lceil b\rceil^c$. Then there exists $a_j\in\support(b)$ such that $q\in\lceil a_j\rceil^c$. Since any $p\leqslant q$ is in both $\lceil b\rceil^c$ and $\lceil a_j\rceil^c$, the power of $x_i$ in $f(b)$ and $f(a_j)$ will be equal.

To show that $f$ is meet-preserving follows from the above and the fact that $\support(p \wedge p') = \support(p) \cap \support(p')$, since  
\begin{equation*}
\begin{aligned}
f(p \wedge p') & = \lcm \{ f(a_i) \,\vert\, a_i \in \support(p \wedge p')\}  \\
                         &=  \lcm \{ f(a_i) \,\vert\, a_i \in \support(p)  \cap \support(p')\} = f(p) \wedge f(p'). 
\end{aligned}
\end{equation*}

Finally, we need to show that this map is injective. We begin by showing that $f(a) \leqslant f(b)$ implies that $a \leqslant b$.  If $f(a) \leqslant f(b)$ then $\mi(P)\cap\lceil a\rceil^c \subset \mi(P)\cap\lceil b\rceil^c$ since all meet-irreducibles must be non trivially labeled.  This is equivalent to saying that  $\mi(P) \cap \lceil{b}\rceil \subset \mi(P) \cap \lceil{a}\rceil$.  All elements in $P$ are determined by their meet-irreducibles, so we get get that $a \leqslant b$.  To see that $f$ is injective consider the fact that $f(a) = f(b)$ implies both that $f(a) \leqslant f(b)$ and $f(b) \leqslant f(a)$ then by the above argument we have that $a \leqslant b$ and $b \leqslant a$ as needed.    
\end{proof}  

To complete our characterization, we also show that every monomial ideal is in fact a coordinatization of its LCM lattice as follows.  Let $M$ be a monomial ideal with $n$ generators and let $P_M$ be its LCM lattice.  For notational purposes, define $P = \{p \, | \, \overline{p} \in P_M\}$ to be the finite atomic lattice where $p < p'$ if and only if $\overline{p} < \overline{p'}$ in $P_M$.  In other words, we simply forget the data of the monomials in $P_M$.  Define a labeling of $P$ in the following way, let $\mathcal{D}$ be the set:
\begin{equation}\label{deficitLabeling}
\{m_p =  \frac{\gcd \{ \overline{t} \,| \, t > p\}}{\overline{p}} \,|\,  p \in P\}.
\end{equation}
Where $\gcd \{ \overline{t} \,| \, t > p\}$ for $p = \hat{1}$ is defined to be $\overline{\hat{1}}$.  Note that the quotients $m_p$ will be a monomial since clearly $\overline{p}$ divides $\overline{t}$ for all $t >p$.

\begin{lemma} \label{deficitsCoverMI} 
If $p$ is a meet-irreducible then $m_p \neq 1$.  In other words, all meet-irreducibles are labeled non-trivially.
\end{lemma}

\begin{proof}
If $p$ is a meet-irreducible then we know that all $t > p$ are greater than or equal to some $p \vee a_i$ where $p \vee a_i$ is the unique element covering $p$.  In particular we know that $a_i$ is not less than $p$.  So, with this knowledge we can easily see that $\overline{p}, \overline{p \vee a_i}$ and $\overline{a_i}$ all divide $\gcd\{ \overline{t} \, | \, t > p\}$.  Thus, since $\overline{a_i}$ does not divide $\overline{p}$ we see that $\overline{p} \neq  \gcd\{ \overline{t} \, | \, t > p\}$ and $m_p \neq 1$ as needed.
\end{proof}

The next lemma implies that variables in this labeling must lie along chains.

\begin{lemma} \label{deficitsVarAlongChains}
If $\gcd\{m_p, m_q\} \neq 1$ then $p$ and $q$ must be comparable. 
\end{lemma}

\begin{proof}
To show this we will show the contrapositive.  Suppose $p$ and $q$ are not comparable.  Then using $\overline{p} = \prod x_{i} ^ {a_i}$ and $\overline{q} = \prod x_{i}^{b_i}$ we can define the following sets:
$$ A = \{ i \, | \, a_i > b_i \}, B = \{i \, | \, b_i > a_i\}, C = \{i \, | \, a_i = b_i \}.$$
Clearly since $\overline{p}$ and $\overline{q}$ do not divide each other we can see that sets $A$ and $B$ are non empty, and by definition their intersections are empty.

Now consider the element $\overline{r} = \lcm(\overline{p}, \overline{q})$ and its corresponding element $r \in P$.  By definition the exponents on $x_i$ in $\overline{r}$ are $a_i$ if $i \in A$ or $i \in C$, and $b_i$ if $i \in B$.  Moreover since $r$ is greater than both $p$ and $q$, we know that both $\gcd\{\overline{t} \, | \, t > p\}$ and $\gcd\{\overline{t} \, | \, t > q\}$ divide $\overline{r}$. 

Now we want to show that the $\gcd\{ m_p, m_q\} = 1$.  For notational purposes let  $\prod x_{i}^{n_i} = \gcd\{ \overline{t} \, | \, t> p\}$.  Then we know that if $i \in A$ or $C$, $n_i = a_i$, and if $i \in B$ then $b_i \geq n_i \geq a_i$.  So when we divide by $\overline{p}$ we are left with a monomial consisting only of variables indexed by $B$.  A similar argument shows that $m_q$ consists only of variables indexed by $A$.  Since $A \cap B = \emptyset$ we have that $\gcd\{ m_p, m_q\} = 1$.
\end{proof}

Using this labeling we can prove our desired result that every monomial ideal can be realized as a coordinatization of its LCM lattice.

\begin{proposition}\label{realizable}
The labeling of $P$ as defined by equation \ref{deficitLabeling} is a coordinatization and the resulting monomial ideal $M_{P, \mathcal{D}} = M$.
\end{proposition}

\begin{proof}
Lemmas \ref{deficitsCoverMI} and \ref{deficitsVarAlongChains} both show that $\mathcal{D}$ is a coordinatization of $P$.  It remains to show that the resulting monomial ideal is equal to $M$.

To show this we will demonstrate that for each atom $$x(a_i) = \prod_{p \in \lceil a_i \rceil^c} m_p$$ divides $\overline{a_i}$ and vice versa.   

To see that $x(a_i) \mid \overline{a_i}$, consider the variable $x$.  We claim that if $p \in \lceil a_i \rceil^c$ and $x^n \mid m_p$ then $x^n \mid \overline{a_i}$.  If $x^n \mid m_p$ then $x^n \overline{p} \mid \overline{t}$ for all $t > p$.  Moreover, $a_i \vee p > p$ since $p$ is not greater than $a_i$  and $\overline{a_i \vee p} = \lcm(\overline{a_i}, \overline{p})$ by construction, so $x^n \overline{p} \mid \lcm(\overline{a_i}, \overline{p})$.  This implies $x^n\mid \overline{a_i}$.  

Additionally, if we have a chain in $P$, $p_1 < \dots < p_k$ such that $p_k \in  \lceil a_i \rceil^c$ and $x^{n_i} \mid m_{p_i}$ then $x^{\sum  n_i} \mid \overline{a_i}$.  This follows from the fact that $x^{n_i} \mid m_{p_i}$ implies that $x^{n_i} \overline{p_i} \mid \overline{p_{i+1}}$ and further implies that $x^{\sum  n_i} \mid \overline{p_k}$.  By the previous paragraph we get the desired result, and successfully prove that $x(a_i) \mid \overline{a_i}$.

To show that $\overline{a_i} \mid x(a_i)$, we first show that if some variable $x$ divides $\overline{a_i}$ that there exists some $p \in \lceil a_i \rceil^c$ such that $x \mid m_p$.  If we let $$p = \bigvee_{\displaystyle{x \nmid \overline{r}}} r,$$ then clearly $x \nmid \bar{p}$ so $p \in \lceil a_i \rceil^c$ but $x$ divides all elements greater than $p$ by construction so $x \mid m_p$.  Now to understand what happens when $x^n \mid \overline{a_i}$, we apply this idea successively by defining elements $$p_i = \bigvee_{\displaystyle{x^{i} \nmid \overline{r}}} r.$$  By inclusion of sets, we see that $p_i \leq p_{i+1}$ so we actually get a chain of elements in $P$.  If some elements in the chain are actually equal, for example $p_{t} = \dots = p_{t+l}$, then by the above reasoning $x^{t} \mid \overline{p_t}$ and $x^{t+l+1} \mid \overline{r}$ for all $r > p_t$ so we may conclude that $m_{p_t}$ is divisible by $x^{l+1}$.  Since all of the elements in this chain are not greater than $a_i$ by construction, this shows that if $x^n \mid \overline{a_i}$ that $x^n \mid x(a_i)$ as needed.  

Therefore for each atom we have that $x(a_i) \mid \overline{a_i}$ and vice versa, thus $M_{P,\mathcal{D}} = M$.

\end{proof}

\begin{figure}[ht]
\centering
\subfigure[Minimal labeling]{
\includegraphics[scale=.5]{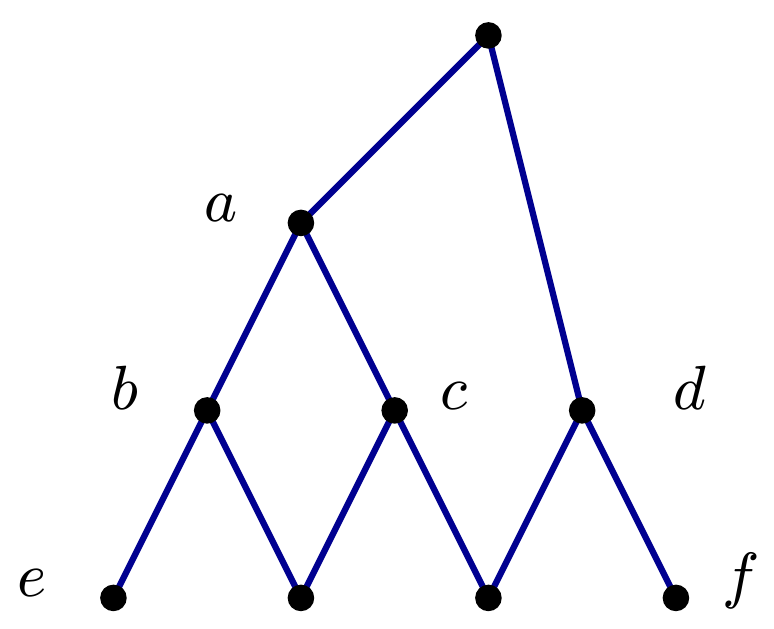}
\label{minSF}
}
\hspace{.5cm}
\subfigure[ECCV labeling]{
\includegraphics[scale=.5]{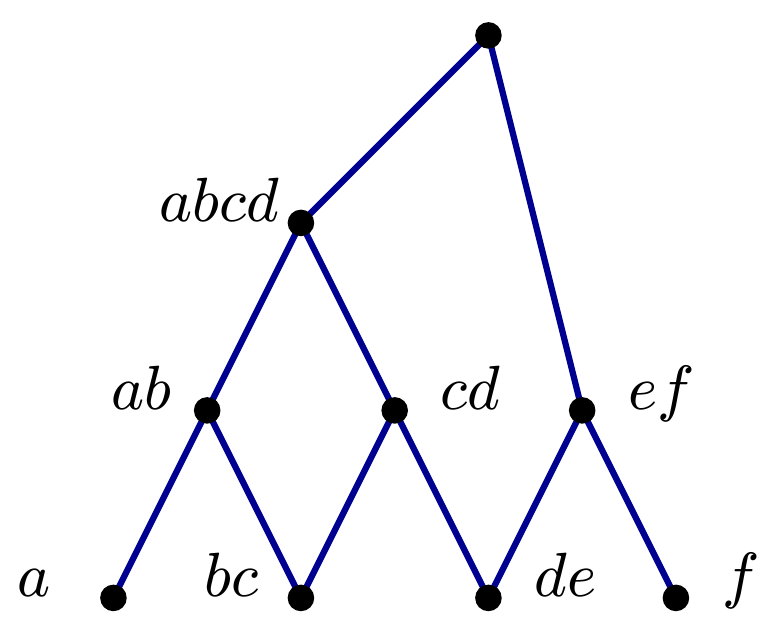}
\label{ECCV}
}
\caption{}
\end{figure}

\begin{example}
The labeling in figure \ref{minSF} obviously satisfies the conditions of \ref{coordinatizations} since only meet-irreducibles are labeled and each variable is used only once.  In fact this is the coordinatization which produces a ``minimal squarefree monomial ideal'' found in \cite{phan}.   The monomial ideal given by this coordinatization in figure \ref{minSF} is $M = (cdf, def, bef, abce)$.
\end{example}

\begin{example}
The labeling in figure \ref{ECCV} will be of use to us later. It is an example of the ECCV labeling (Every Chain Covered by a Variable) which is defined as follows.  Let $\{c_1, \dots, c_t\}$ be the set of all maximal chains in $P$.  Then for variables in the ring $R = k[x_1,\dots x_t]$ define the following labeling, 
	\[\mathcal{M} = \{m_p = \prod_{i\,:\,p \in c_i} x_i \,|\, p \in P \}.\]

Every meet-irreducible is covered since every element of $P$ is covered and each variable appears only along one chain by definition, so the conditions of \ref{coordinatizations} are satisfied.  The example in figure \ref{ECCV} shows such a coordinatization where the monomial ideal is $$M = (bc^2d^2e^2f^2, ade^2f^2, a^2b^2cf, a^3b^3c^3d^3e).$$
\end{example}

The construction of ``nearly Scarf'' monomial ideals found in \cite{velascoFrames} and \cite{velascoNonCW} can easily be identified as a specific coordinatization where every element in $P- \{\hat{0}, \hat{1}\}$ is labeled.  Additionally, the construction of monomial ideals whose minimal resolutions are supported on trees in \cite{floystad} can be seen as an instance of coordinatizing a specific lattice defined in terms of the tree (see section 6.2 in \cite{mapesThesis}).

\section{The set of all atomic lattices with $n$ atoms}\label{defoSection}
The set of all possible deformations of the exponents of a monomial ideal can be viewed as a fan in exponent space.  This set can be unwieldy if the monomial ideal one starts with is far from being generic, and has lots of redundancy in the combinatorial type of the monomial ideals occuring as deformations.  Thus, it can be easier to consider questions about these deformations in the following setting.  Define the set $\mathcal{L}(n)$ to be the set of all finite atomic lattices with $n$ ordered atoms.  This set has a partial order where $Q \leqslant P$ if and only if there exists a join-preserving map which is a bijection on atoms from $P$ to $Q$ (note that such a map will also be surjective).  In \cite{phan}, Phan shows the following result.

\begin{theorem}\label{Ln} 
With the partial order $\leqslant$, $\mathcal{L}(n)$ is a finite atomic lattice with $2^n -n -2$ atoms.
\end{theorem}

We include a proof here since it currently does not exist in the literature.

\begin{proof}
It is clear that $\mathcal{L}(n)$ has both a unique minimal and maximal element, the lattice where the atoms are also the coatoms, and the boolean lattice respectively.  Thus by Proposition~\ref{prop:Stanley} in order to show that $\mathcal{L}(n)$ is a lattice it is enough to show that $\mathcal{L}(n)$ is a meet semilattice.   

Here it will be convenient to think of the elements in $\mathcal{L}(n)$ as sets $\mathcal{S}$ as described in Section \ref{background}.  For any two elements $\mathcal{S}_1$ and $\mathcal{S}_2$ in $\mathcal{L}(n)$ define their meet to be the intersection of the two sets, denoted $\mathcal{R}$.  Clearly $\mathcal{R}$ is the largest set contained in both $\mathcal{S}_1$ and $\mathcal{S}_2$ thus we only need to show that $\mathcal{R}$ is an element of $\mathcal{L}(n)$.  Clearly $\mathcal{R}$ contains the sets $\{1, \dots, n\}, \emptyset, \{i\}$ for all $i$ since all of these are contained in $\mathcal{S}_1$ and $\mathcal{S}_2$.  Lastly, it is closed under intersections, since both $\mathcal{S}_1$ and $\mathcal{S}_2$ are closed under intersections.

To see that $\mathcal{L}(n)$ is a finite atomic lattice, observe that the elements covering the minimal element are of the form \[\mathcal{S}= \{\emptyset, \{1\}, \{2\}, \dots, \{n\}, \{1,\dots,n\}, \sigma\}\] where $\sigma$ is any subset of $\{1,\dots,n\}$ other than those already in $\mathcal{S}$.  Clearly there are $2^n - n -2$ such subsets so that is the number of atoms.  Moreover, any element $L \in \mathcal{L}(n) - \hat{0}$ is the join of atoms.  
\end{proof}  

A {\it deformation} of a monomial ideal $M = (m_1,\dots,m_t)$ is a choice of vectors $\{\epsilon_1,\dots, \epsilon_t\}$ where each $\epsilon_i \in \mathbb{R}^n$ (where $n$ is the number of variables), and the following condition is satisfied: \[ m_{is} < m_{js} \mbox{ implies } m_{is}+\epsilon_{is} < m_{js} + \epsilon_{js} \mbox{, and }\]\[ m_{is} = 0 \mbox{ implies } \epsilon_{is} = 0. \]  Here $m_{is}$ is the exponent of $x_x$ in $m_i$.

It is easy to see that for any given monomial ideal $M$ and a deformation of exponents $M_{\epsilon}$, the LCM lattice of $M_{\epsilon}$ is greater than that of $M$ in $\mathcal{L}(n)$ (see example 3.4 (a) in \cite{GPW}).  The converse, which is the following theorem, indicates that deformations can be as ``badly'' behaved as possible.

\begin{theorem}\label{chains=defos}
If $P \geqslant Q$ in $\mathcal{L}(n)$ then there exists a  coordinatization of $Q$ such that via deformation of exponents one can obtain a coordinatization of $P$.
\end{theorem}

We will need the following lemma in order to simplify the proof of Theorem \ref{chains=defos}.

\begin{lemma}\label{mapsOnChains}
Let $f: P\rightarrow Q$ be a join preserving map  between finite atomic lattices which is a bijection on atoms.  If $c$ is a chain in $P$ then $f(c)$ is a chain in $Q$.
\end{lemma}

\begin{proof}
If $c$ is a chain in $P$ then $c =\{ \hat{0}, a_i, p_1, \dots ,\hat{1}\}$, where the elements are listed in increasing order.  Since $f$ is join preserving and thus order preserving, the image of these elements will be a (possibly shorter) chain in $Q$.
\end{proof} 

Now for the proof of the theorem.

\begin{proof}
First, label $P$ with the ECCV labeling from above.  Then construct a labeling of $Q$ as follows.  Since $P \geqslant Q$ there is a join preserving map $f : P \rightarrow Q$.  To each element $q \in Q$ assign the monomial \[ \prod_{j \in I} x_j, \] where $I = \{j \, | \, x_j \mbox{ divides }m_p \mbox{ for all } p \in f^{-1}(q)\}$.  To see that this is in fact a coordinatization we need to check the two conditions in Theorem \ref{coordinatizations}.  Clearly all meet-irreducibles in $Q$ will be labeled, and since variables $x_j$ only appeared along chains in $P$, in this new labeing they will only appear along the image of that chain under $f$ which by Lemma \ref{mapsOnChains} will also be a chain in $Q$.  Thus, this labeling is in fact a coordinatization.

It remains to show that there exists $\epsilon_i$ for each of the $n$ atoms, such that the monomial ideal obtained for $P$ is a deformation of exponents for the monomial ideal obtained for $Q$ (with these coordinatizations).  We do this by considering chains in both $P$ and $Q$ and their relation to each other under the map $f$.  

Let $c_j$ be the chain in $P$ which is labeled by the variable $x_j$ under the greedy labeling.  Note that we can write the monomial associated to an atom $a_i$ as follows \[ \prod_j \prod_{p\, \in \lceil{a_i}\rceil_P^c \cap c_j}  x_j, \] where the subscript $P$ indicates that both the order ideal and the complement are in $P$.  Similarly with the coordinatization of $Q$ given above we can think of the monomials for the lattice $Q$ as \[ \prod_j \prod_{q \,\in \lceil{a_i}\rceil_Q^c \cap f(c_j)}  x_j. \]  Given these descriptions of the monomials generators for each lattice makes it clear that for $\epsilon_i$ we want to define $$\epsilon_{ij} = \vert \lceil{a_i}\rceil_P ^c \cap c_j \vert - \vert \lceil{a_i}\rceil_Q^c \cap f(c_j)\vert. $$  
Observe that if $\lceil{a_i}\rceil_Q^c \cap f(c_j) = \emptyset$ then every $q \in f(c_j)$ is greater than $a_i$ (in $Q$).  This implies that $a_i$ is an element in the chain $f(c_j)$ and since $f$ is a bijection on atoms this means that $a_i$ is an element of the chain $c_j$ in $P$.  Thus $\lceil{a_i}\rceil_P^c \cap c_j = \emptyset$ also.  Thus if the exponent on $x_j$ for the monomial $m_i$ is $0$ then $\epsilon_{ij} = 0$.  This collection of $\epsilon_{ij}$ gives a deformation of exponents from a monomial ideal whose LCM lattice is $Q$ to one whose LCM lattice is $P$.  
\end{proof}

Note that this proof uses the fact that we can represent any deformation of exponents using integer vectors rather than working with real exponents.  An obvious corollary to this is the following.

\begin{corollary}
Given a lattice $Q \in \mathcal{L}(n)$ there exists a coordinatization for which every element in $\lceil{Q}\rceil$ is the LCM lattice of a deformation of exponents of that coordinatization.
\end{corollary}

\begin{proof}
Apply the coordinatization used to prove Theorem \ref{chains=defos} where $Q$ is your given lattice and $P = B_n$.  Use the same coordinatization given in the proof for every element $P' \in \lceil{Q}\rceil$.  Now it remains to show that with these coordinatizations of $P'$ and of $Q$, we can find an $\epsilon_i$ for each of the $n$ atoms giving a deformation of exponents for the monomial ideal obtained for $Q$.  Again for a variable $x_j$ which appears only along one chain $c_j$ we define $\epsilon_{ij}$ to be $\vert \lceil{a_i}\rceil_{P'} ^c \cap c_j \vert - \vert \lceil{a_i}\rceil_Q^c \cap f(c_j)\vert$.
\end{proof}

Together these results give evidence that deformations of exponents might be best understood in the context of the set $\mathcal{L}(n)$.  In particular given an ``abstract monomial ideal'' $L$, the set of all possible deformations is the set $\lceil L \rceil$.  However for a given coordinatization the set $\lceil L \rceil$ may be much larger than the actual set of deformations that occur, nevertheless it gives a good starting point.  Alternatively one could view all possible deformations for a given monomial ideal as a fan in exponent space.  Specifically if $M$ is in the polynomial ring with $n$ variables and has $k$ monomial generators then it sits in the intersection of hyperplanes in $\mathbb{R}^{mk}$ which are defined by multiple generators having the same exponent on a variable.  In this setting, deformations correspond to moving off these intersections, but various partial deformations might yield a combinatorially equivalent monomial ideal.  If we compare these two approaches we can quickly see that the set $\lceil L \rceil \in \mathcal{L}(n)$ only lists possible deformations up to combinatorial type thus making it a simpler set to study than the fan.

\section{Structure of $\mathcal{L}(n)$}\label{structure}
In this section, we study some basic properties of $\mathcal{L}(n)$.  Counting arguments show that $\vert \mathcal{L}(3) \vert = 8$ and $\vert \mathcal{L}(4) \vert = 545$, and by using a reverse search algorithm on a computer one can see that $\vert \mathcal{L}(5)\vert = 702,\!525$ and $\vert \mathcal{L}(6)\vert = 66,\!960,\!965,\!307$ (see appendix A in \cite{mapesThesis}).  Thus the complexity of $\mathcal{L}(n)$ rapidly increases with $n$.  Still there are nice properties that we can show about $\mathcal{L}(n)$ which give it some extra structure.

Most importantly it is necessary to first understand what covering relations look like in $\mathcal{L}(n)$.

\begin{proposition}\label{coveringRelations}
If $P \geqslant Q$ in $\mathcal{L}(n)$ then $P$ covers $Q$ if and only if $\vert P \vert = \vert Q \vert +1$.
\end{proposition}

\begin{proof}
Clearly $P \geqslant Q$ which implies that $\mathcal{S}_Q \subset \mathcal{S}_P$.  Since every set in $\mathcal{S}$ corresponds to an element in the associated lattice, this implies that $\vert P \vert \geqslant \vert Q \vert$.  It remains to show that they differ by one element when $P$ covers $Q$.  

Suppose they differ by $2$ elements, then $\mathcal{S}_P = \mathcal{S}_Q \cup \{\sigma, \beta\}$ where $\sigma$ and $\beta$ are subsets of $\{1,\dots,n\}$ satisfying the conditions that $\sigma \cap \beta$, $\sigma \cap s$, and $\beta \cap s$ are either in $\mathcal{S}_Q$ or $\{\sigma, \beta\}$ for all $s \in \mathcal{S}_Q$.  The argument is that in this case $P$ cannot cover $Q$ as there exists a lattice $T \not= P$ satisfying $P \geqslant T \geqslant Q$.  Let $\mathcal{S}_T$ be one of the following  
\begin{itemize}
\item $\mathcal{S}_Q \cup \{\sigma \cap \beta \}$ if $\sigma \cap \beta \not \in \mathcal{S}_Q$ 
\item $\mathcal{S}_T = \mathcal{S}_Q \cup \{\sigma \}$ if $\sigma \subset \beta$
\item $\mathcal{S}_T = \mathcal{S}_Q \cup \{ \beta \}$ if $\beta \subset \sigma$
\item $\mathcal{S}_T = \mathcal{S}_Q \cup \{\sigma\}$ or $\mathcal{S}_T = \mathcal{S}_Q \cup \{\beta \}$ if $\sigma$ and $\beta$ are not subsets of each other.
\end{itemize}

In any of these cases, $T \geqslant Q$ and $\vert T \vert = \vert Q \vert +1$.

\end{proof}  

The upshot of Proposition \ref{coveringRelations} is the next nice result.  It is easy to see that $\mathcal{L}(3) = B_3$ (the boolean lattice on 3 atoms), whereas $\mathcal{L}(4) \not = B_4$ (and the latter is true for all $n \geqslant 4$ by Proposition \ref{meetIrredLn}).  However, one can ask, what if any are the nice properties of $B_n$ that are retained by $\mathcal{L}(n)$.  One answer is the following theorem.  

\begin{theorem}\label{gradedLn}
$\mathcal{L}(n)$ is a graded lattice of rank $2^n -n -2$, ie. this is the length of all maximal chains. 
\end{theorem}

\begin{proof}
The maximal element of $\mathcal{L}(n)$ is the lattice $B_n$ and $\vert B_n \vert = 2^n$.  The minimal element of $\mathcal{L}(n)$ is the unique lattice on $n$ atoms where the atoms are also the coatoms, it has $n+2$ elements.  Then by \ref{coveringRelations} every maximal chain in $\mathcal{L}(n)$ has length $2^n -(n+2)$ and so it is graded of rank $2^n-n-2$.  
\end{proof}

It follows from Theorem \ref{gradedLn} that if $\mathcal{L}(n)$ is co-atomic then it will be isomorphic to $B_n$.  With the following description of the meet-irreducibles it is easy to see that the only case where this happens is for $n=3$.

\begin{proposition}\label{meetIrredLn}
The number of meet irreducibles in $\mathcal{L}(n)$ is $$n(2^{n-1} -n).$$
\end{proposition}

\begin{proof}

A lattice $P \in \mathcal{L}(n)$ is meet irreducible if it is covered by only one element, call it $Q$.  Thinking of $P$ as a collection of subsets $\mathcal{S}_P$ then by Proposition \ref{coveringRelations} that means that there is only one way to add a subset to $\mathcal{S}_P$ so that the resulting set of sets is closed under intersections.  

Let $X$ be the set of all lattices $L_{i,\sigma} \in \mathcal{L}(n)$ where $\vert \sigma \vert \geqslant 2$, $1 \leqslant i \leqslant n$  and  $L_{i,\sigma} = B_n - [\sigma, \{1,\dots ,\hat{i},\dots,n\}]$.  To see that  $L_{i,\sigma} \in \mathcal{L}(n)$ first observe that it has $n$ atoms, a maximal element and a minimal element.  It remains to show that it is a meet-semilattice.  To show this assume that there are two elements $a,b$ in $L_{i, \sigma}$ which do not have a meet,  i.e. their meet was in $ [\sigma, \{1,\dots ,\hat{i},\dots,n]\}$.  If this is the case then both $a \geqslant \sigma$ and $b \geqslant \sigma$ but both are not comparable to $\{1,\dots ,\hat{i},\dots,n\}$, so $i \in a$ and $i \in b$.  This means that the meet of $a$ and $b$ should contain $\sigma$ and $\{i\}$ thus their meet cannot be in $[\sigma, \{1,\dots ,\hat{i},\dots,n\}]$.  So $L_{i,\sigma}$ is a meet semi-lattice.

Now we need to show that $X = \mi(\mathcal{L}(n))$.  To see that $X \subset \mi(\mathcal{L}(n))$ observe that the only candidates for elements to add to $L_{i,\sigma}$ come from $[\sigma, \{1,\dots ,\hat{i},\dots,n\}]$.  Choose such a $\gamma$ such that $\sigma \subset \gamma$.  Observe that $\sigma \cup \{i\} \in L_{i,\sigma}$ from the above discussion.  So $\gamma \cap (\sigma \cup \{i\}) = \sigma$ which shows that if one cannot only add $\gamma$ to $L_{i, \sigma}$ and still have a set which is closed under intersections.  Thus the only element covering $L_{i, \sigma}$ is $L_{i,\sigma}\cup \sigma$.

To see that $\mi(\mathcal{L}(n)) \subset X$, let $L \in \mi(\mathcal{L}(n))$ this means that there is only one element $L' = L \cup \{\sigma\}$ which covers $L$.  From this we can deduce that at minimum $L$ contains a minimal element, a maximal element and all $n$ atoms and does not include $\sigma$.  Furthermore since this is the only covering element, we can deduce that for every subset $\gamma \subset \{1,\dots,n\}$ not in $L$ and not equal to $\sigma$ there exists a subset $\alpha \subset \{1,..,n\}$ in $L$ such that $\gamma \cap \alpha$ is not in $L$.  Thus adding $\gamma$ to $L$ would not produce a covering relation.  To see what conditions this forces on $L$ consider three types of subsets $\gamma$:

\vskip.2in
\noindent{\bf Type 1:  } $\gamma$ is not comparable to $ \sigma$ \\
{\bf Type 2:  } $\gamma \subset \sigma$\\
{\bf Type 3:   } $\sigma \subset \gamma$\\

Considering the subsets $\gamma$ of type 1 of size 2, it is clear that all of them must be elements of $L$ otherwise $L \cup \{\gamma\}$ would also cover $L$.  Once we have all the subsets of size 2, the argument applies for all subsets of type 1 with size 3, and so on until $n-1$.  So now $L$ must consist of all of $\hat{0}, \hat{1}$, the atoms, and all subsets of type 1.

Now, with all the subsets of type 1 in L consider the fact that $L'$ needs to cover $L$ or in other words $\sigma$ intersect any subset of type 1 needs to already be in $L$.  Since we have all subsets that do not include or not included by $\sigma$ this means that we need all subsets of sigma in $L$.  So, now $L$ must consist of $\hat{0}, \hat{1}$, the atoms, and all subsets of type 1 and 2.

Finally we need to consider subsets of type 3.  Again if we start by considering such subsets which differ from $\sigma$ by only one element (i..e $\gamma = \sigma \cup \{ i\}$) it is easy to see that only one such $\gamma$ can be included in $L$ since if there were more than one then $L$ would not be closed under intersections since $\sigma \not \in L$.  Moreover, one such element must be included in $L$ otherwise there are more than one lattice covering $L$.  Similar arguments apply for $\gamma'$ where $\sigma \subset \gamma'$ and $\vert \gamma' \vert \geqslant \vert \sigma \vert +2$.  Although here it is necessary that if $\vert \gamma' \vert = \vert \sigma \vert +2$ that only the $\gamma'$ which include the $\gamma = \sigma \cup \{i\}$ already included into $L$ be included as well, and so on.  So, we see that all of the elements so all of the elements of type 3 that we've included contain $i$, so it must be true that we've included everything in $B_n$ except for the elements greater than $\sigma$ and less than $\{1,\dots, \hat{i},\dots,n\}$ since these are precisely the elements of type three not containing $i$.   Thus we have shown that $\mi(\mathcal{L}(n)) \subset X$.   

Now, all that remains is to count the number of elements in $X$.  First fix $i$, there we see that the number of possible $L_{i,\sigma}$ lattices is $2^{n-1} - n$.  This is because the number of subsets of size 2 or more in a set of $n-1$ elements is $2^{n-1}-(n-1)-1$.  Finally we see that $\vert X \vert = n(2^{n-1} -n)$ by letting $i$ range from 1 to $n$.
\end{proof}

\section{Relationship to cellular resolutions}\label{cellResSect}

In this section we prove the main results of this paper which focus on how the geometric structure of  minimal resolutions stabilize as one moves up chains in a fixed total Betti stratum.  The approach is to first understand the simpler but well understood cases where a minimal resolution can be supported on a simplicial complex, in particular the Scarf complex.  We then show how theorems of this type can be generalized to non simplicial complexes.  First we must introduce some semantics and notation.  

Due to the formulas for multigraded Betti numbers in terms of order complexes of intervals in the LCM lattice(see Theorem 2.1 in \cite{GPW}), we can refer to the Betti numbers of a finite atomic lattice as opposed to the Betti numbers of a monomial ideal.  Specifically for $p \in \LCM(M)$ the formula for computing multigraded Betti numbers in homological degree $i$ is:
$$b_{i,p}(R/M) = \dim \tilde{H}_{i-2}(\Delta (\hat{0},p); k),$$ where $M$ is a monomial ideal whose LCM lattice is $P$, and $\Delta (\hat{0},p)$ is the order complex of the open interval from $\hat{0}$ to $p$.  Note also that we can substitute the cross cut complex $\Gamma$ defined in section \ref{background} for $\Delta$.  Since all monomial ideals with the same LCM lattice have isomorphic minimal resolutions this means that if a cell complex supports the minimal resolution of one ideal then it will support the minimal resolution for all possible coordinatizations of the LCM lattice of that ideal.  In this sense, we will discuss the cellular resolution of a finite atomic lattice.  

Notice that in a finite atomic lattice, any given element is not necessarily expressed as a unique join of atoms.  For example the join of atoms 1 and 2 may be equal to the join of atoms 1, 2, and 3.  To account for this ambiguity it will be useful for us to introduce the following set associated to every element in the lattice, define $\equivSet_P(p)$ to be the set containing all subsets of the atoms whose join in $P$ is equal to $p$.  

It is now easy to define the Scarf complex of a monomial ideal $M_P$ in terms of the lattice $P$, \[\Scarf(P) = \Gamma (\{p \in P \, | \, | \equivSet_P(p)| = 1 \}) \subset \Gamma(P).\]  Note that because $P$ is an atomic lattice $\Gamma (P)$ is the Taylor Complex associated to $P$.

By results of \cite{bps} and \cite{msy} if a given monomial ideal is generic or strongly generic then its minimal resolution is the Scarf complex.  A strongly generic monomial ideal is one where no variable appears with the same exponent in two or more generators.  The weaker condition of being generic implies that if two generators have the same exponent in a variable that there is a third generator which strictly divides their lcm.  Note however, that there may be monomial ideals whose minimal resolution is the Scarf complex, yet the ideal is not generic.  An obvious example of this phenomenon is if one takes a generic monomial ideal and polarizes to obtain a squarefree monomial ideal.  It will have the same Scarf complex which supports the minimal resolution since LCM lattices are preserved under polarization.  It is rare however, for squarefree monomial ideals to be generic since all variables always appear with the same exponent.  

However there exist examples which are not simply a polarization of a generic monomial ideal, whose minimal resolution is supported by the Scarf complex, but they are not generic or strongly generic.  The following example of an abstract monomial ideal illustrates well an example of a monomial ideal whose minimal resolution is Scarf, but where there is no possible generic coordinatization.  

\begin{example}\label{scarfNotGeneric}
The lattice $P$ in figure \ref{scarfNotgen} is the {\it augmented face lattice} (face poset, which in this case is a meet-semilattice, plus a maximal element) of a simplicial complex consisting of $4$ vertices and $3$ edges.  Every point in $P$ except for the minimal and maximal elements represents a multidegree that has a nonzero betti number.  This is easy to see since for all of the atoms $\tilde{h}_{-1}(\Gamma(P_{<a_i}), k) = 1$ and for each element $p$ covering an atom $\Gamma(P_{<p})$ consists of two vertices thus $\tilde{h}_0 = 1$.  Thus, $P$ will always be resolved by its Scarf complex.

\begin{figure}
\center
\includegraphics[scale=0.5]{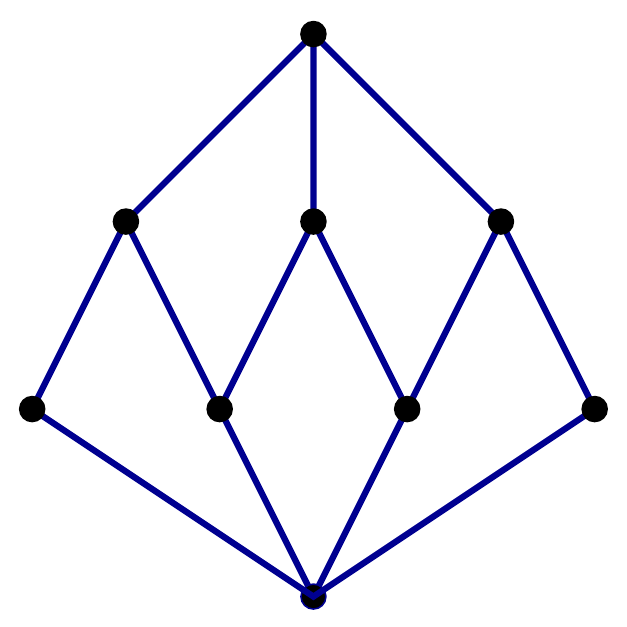}
\caption{Lattice resolved by Scarf complex which has no generic coordinatization}\label{scarfNotgen}
\end{figure}

\noindent Note however, that every possible coordinatization of $P$ fails to satisfy the conditions to be generic.  In other words, that for every coordinatization a variable appears with the same non-zero exponent in either $m_1$ and $m_2$, $m_2$ and $m_3$, or in $m_3$ and $m_4$.  Since  the meet-irreducibles of $P$ are precisely $p_{12} = a_1 \vee a_2$, $p_{23} = a_2 \vee a_3$, $p_{34} = a_3 \vee a_4$, $a_1$, and $a_4$ any coordinatization must cover these.  I will focus on just showing that for all coordinatizations there is a variable appearing with the same non-zero exponent for the pair $m_1$ and $m_2$.  Note that the monomial $m_1$ is determined by the labelings found on $\{p_{23}, p_{34}, a_2, a_3, a_4\}$ and $m_2$ is determined by $\{p_{34}, a_1, a_3, a_4\}$.  Moreover, $p_{34}$ and $a_4$ must be labeled and if $a_3$ happens to be non-trivially labeled it also appears in both of these sets.  Any variable appearing in the label on $p_{34}, a_3,$ and $a_4$ cannot appear anywhere else, thus it must appear with the same exponent in both $m_1$ and $m_2$ and no other generator divides the lcm associated to element $p_{12}$ by construction.  So there is no generic coordinatization of $P$.

\end{example}

Understanding when generic coordinatizations exist must be determined on a case by case basis.  For strongly generic coordinatizations however,  the following theorem characterizes precisely when they exist.  
 
 \begin{theorem}\label{genericCoords}
$P$ in $\mathcal{L}(n)$ admits a strongly generic coordinatization if and only if $P$ is a graded finite atomic lattice whose chains have rank $n$.
\end{theorem}

\begin{proof}
To show that such graded lattices have strongly generic coordinatizations, coordinatize $P$ using the ECCV coordinatization.  Now, we just need to show that the resulting monomial ideal is strongly generic.  In other words we need to show that for any variable $x_j$, if it appears in any two monomials $m_{a_i}$ and $m_{a_k}$ that it has a different exponent.  By the definition of the monomials $m_{a_i}$ and $m_{a_k}$ this amounts to showing that the intersections of the complements of the filters $\lceil{a_i}\rceil$ and $\lceil{a_k}\rceil$ with the chain corresponding to $x_j$ are different.

Let $c_j = \{\hat{0}, a_{1_j}, a_{1_j} \vee a_{2_j}, \dots, a_{1_j} \vee a_{2_j} \vee \dots \vee a_{n_j} = \hat{1} \}$ be the chain which is entirely labeled by the variable $x_j$.  Since $P$ is graded of rank $n$ we know that there are $n+1$ elements in $c_j$ and that the $i$-th element in the chain is the join of $i-1$ atoms of $P$.  Every set $\lceil{a_i}\rceil$ intersects $c_j$ at a different spot along the chain, so it is likewise with the complements of these filters.  This guarantees that each variable $x_j$ appears with a different exponent in each monomial generator.  

Conversely suppose $M = (m_1,\dots, m_n)$ is a strongly generic monomial ideal. We need to show that its LCM lattice is graded and that the chains have rank n.  Since $M$ is strongly generic then we know that  each variable $x_i$ must achieve its maximum degree $d_i$ for exactly one monomial generator $m_j$.  Let $p = \lcm(m_1,\dots, \hat{m_j},\dots, m_n)$ i.e. omitting generator $j$. Then we know that the degree of $x_i$ in the monomial $p$ must be less than $d_i$.  Thus $p \neq \hat{1}$ in $\LCM(M)$ but $p \wedge a_j = \hat{1}$ where $a_j$ is the atom corresponding to the $j$th generator, so we can conclude that $p$ is covered by $\hat{1}$ and thus is a coatom.

Now, consider $\lfloor p \rfloor \in \LCM(M)$.  This will be $\LCM(M')$ where $M' = (m_1,\dots, \hat{m_j},\dots, m_n)$ which is also a strongly generic monomial ideal.  Repeating the above argument for perhaps a different variable appearing in the generators of $M'$ we can find a $p'$ such that the maximal element of $\LCM(M')$ covers it, i.e. $p$ covers $p'$.  Iterating this process we can find precisely $n$ monomial ideals $M^{(i)}$ so that $M^{(i+1)} \subset M^{(i)}$ and the maximal element in $M^{(i+1)}$ corresponds to a coatom of $M^{(i)}$.  This produces a chain of length $n+1$ connecting $\hat{1}$ to $\hat{0}$ in $\LCM(M)$ and varying over all choices at each iteration we get all chains of $\LCM(M)$.             
\end{proof}

The following statement appeared as a Proposition 5.14 in \cite{phan} and separately as a Theorem 5.3 in \cite{velascoFrames}.  An idea of how to prove this can be summarized as taking the face poset of the simplicial complex, making it a lattice by adding a maximal element (if necessary), and then coordinatizing the lattice.  

\begin{proposition}\label{phanScarf}
Every simplicial complex $X$ not equal to the boundary of a simplex is the Scarf complex of some squarefree monomial ideal.  If $X$ is acyclic then $X$ supports the minimal resolution of that ideal.
\end{proposition}

The intention of the next proposition is to demonstrate that each strata in $\mathcal{L}(n)$ with fixed total Betti numbers has entire regions of lattices whose minimal resolution is the supported by the appropriate Scarf complex.   

\begin{proposition}\label{simplicialFilter}
Suppose $P$ is minimally resolved by its Scarf complex.   If $Q$ is such that  $Q \geqslant P$ in $\mathcal{L}(n)$ and the total Betti numbers of $Q$ are the same as the total Betti numbers for $P$, then $Q$ is also resolved by its Scarf complex which will equal the Scarf complex of $P$.
\end{proposition}

\begin{proof} 
If $P$ is minimally resolved by its Scarf complex, then $\beta_{i,p} \neq 0$ implies that $\beta_{i,p} = \tilde{h}_{i-2}(\Gamma(P_{<p}, k)) = 1$ and $| \equivSet_P(p)| = 1$.  So we need to show the same is true for the appropriate elements $q \in Q$.   Since $Q \geqslant P$ there is a join preserving map $\psi: Q \rightarrow P$ and recall that \[\equivSet_P(p) = \bigcup_{q \in \psi^{-1}(p)} \equivSet_Q(q).\]  Since any $p \in P$ contributing to a Betti number satisfies $| \equivSet_P(p)| = 1$ we see that there must be only one $q \in \psi^{-1}(p)$.  It remains to show that for this $q = \psi^{-1}(p)$ we have $\tilde{h}_{i-2}(\Gamma(Q_{<q}, k)) = 1$.

To see this last fact, consider the fact that since $P$ is resolved by its Scarf complex we can inductively apply the above argument to every single face of $\Gamma(P_{<p})$ to see that $\Gamma(Q_{<q})$ is precisely the same simplicial complex.   Thus it has the same reduced homology. 

Clearly no other elements of $Q$ have nonzero Betti numbers since $Q$ is assumed to be in the same strata of total Betti numbers as P.  Thus, $\Scarf(P) = \Scarf(Q)$ and both are minimally resolved by their Scarf complexes.   

\end{proof}

We can also begin to generalize the previous results to account for ideals whose minimal resolutions will not be supported on a simplicial complex.  Then the following propositions provide a generalization of Proposition \ref{phanScarf} and Proposition \ref{simplicialFilter}.

\begin{proposition}\label{generalizationScarf}
Let $X$ be a regular cell complex such that $X$ is acyclic, and the augmented face poset $P_X$ of $X$ is a finite atomic lattice on $|X^0| = n$ atoms (i.e. the face poset is a meet-semilattice with $n$ atoms).  Then the minimal resolution of any coordinatization of $P$ is supported on $X$.
\end{proposition}

\begin{proof}
Observe that if $P_X$ is the face lattice of $X$ then labeling $X$ with the monomials in any coordinatization of $P_X$ as prescribed in \cite{bs} simply puts the monomial at a point $p \in P_X$ on its corresponding face in $X$.  Moreover, each face in $X$ has a distinct multidegree labeling it.  To show that the resolution of any coordinatization is supported on $X$ we simply need to show that $X_{\leqslant p}$ is acyclic.  This is true by construction though since $X_{\leqslant p}$ corresponds to the $d$-cell that $p$ represents and its boundary.  

Note also that $\Delta({P_X}_{<p})$ is the barycentric subdivision of $X_{<p}$.  In particular, we can easily see that $\tilde{h}_i (\Delta({P_X}_{<p},k) = 1$ for $i = d-1$ where $p$ corresponds to a $d$-cell in $X$ since $X_{<p}$ is the boundary of that $d$-cell.  Thus $\beta_{d+1,p} =1$ for a $d$-cell $F_p$ in $X$.  

If $P$ is the augmented face lattice of $X$ (i.e. if the $\hat{1}$ element actually needs to be added in) then the above description applies to all $p \in P-\hat{1}$.  For $p = \hat{1}$, we have by construction that $X_{< p} = X$ and since $X$ was assumed to be acyclic, $X$ still supports the resolution.   

Moreover, only the multidegrees corresponding to each $p \in P$ can possibly have $\beta_{i,p}$ nonzero so since we've considered all such $p$'s, $X$ supports the resolution of $P$.  This resolution is minimal since $\beta_{i,p} = 1$ for only one $i$, i.e. no map has an integer as an entry of the matrix.
\end{proof}

The following result is a partial generalization of \ref{simplicialFilter}.

\begin{proposition}\label{GenSimpFilt}
Let $P_X$ be as in proposition \ref{generalizationScarf}. Suppose that the total Betti numbers of $Q$ are the same as the total Betti numbers for $P_X$  and that $Q$ satisfies the following two conditions:
\begin{enumerate}
\item $Q$ covers $P_X$
\item $\beta_{i,q} = 1$ for $q = \max (\psi^{-1}(p))$ for any $p \in P_X$ where $\psi$ is the join-preserving map from $Q$ to $P_X$, and $\beta_{i,q} = 0$ otherwise.
\end{enumerate}
Then $Q$ has a minimal resolution supported on $X$.
\end{proposition}

\begin{proof}
Since $Q \in \phi^{-1}({\bf b}_{P_X})$ we know that the total Betti numbers of the minimal resolution of $Q$ are the same as those of $P_X$.  So, all that needs to be shown is that $X$ supports a resolution of $Q$, and then since that resolution has the right total Betti numbers it must be minimal.  Thus, we just need to show that $X_{\leqslant q}$ is acyclic for all $q \in Q$.

One can observe that \[\equivSet_{P_X}(p) = \bigcup_{q \in \psi^{-1}(p)} \equivSet_Q(q). \]   Moreover, $\support(p)$ is the maximal element in $\equivSet_{P_X}(p)$ when it is ordered by inclusion which implies that $\support(p)$ is the maximal element in $\bigcup_{q \in \psi^{-1}(p)} \equivSet_Q(q)$.  So we can observe that there are two types of elements $q \in Q$:  one being where  \[ q = \bigvee_{i \in \support(p)} a_i \] for some $p \in P_X$ where the join is in $Q$;  and the second being the $q \in Q$ that are not of the first type.  Clearly the elements of type one are maximal in the appropriate $\phi^{-1}(p)$.  Note also that since $Q$ covers $P_X$ there is precisely one element of type two in $Q$ due to Proposition \ref{coveringRelations}.  

Now we need to label $X$ with the appropriate multidegrees.  Notice that each face of $X$ corresponding to $p \in P_X$ will be labeled with the appropriate $q \in Q$ of type one which is the maximal element in $\phi^{-1}(p)$.  Moreover, no elements in $Q$ of type two label any faces in $X$.  Thus when we examine the complexes $X_{\leqslant q}$ for $q$ of type one they will all be acyclic for the same reasons as in the proof of Proposition \ref{generalizationScarf}.

It remains to show that for the one element $q \in Q$ of type two we have that $X_{\leqslant q}$ is acyclic.  We know that since $\beta_{i,q} = 0$ for this element, $\Delta(Q_{<q})$ is acyclic.  Moreover, $q$ is the only element in $Q$ which does not correspond to a face in $X$, so $Q_{<q}$ is equal to the face lattice of $X_{<q}$.  This shows that $X_{<q}$ is acyclic since $\Delta(Q_{<q})$ is homotopy equivalent to $X_{<q}$ by barycentric subdivision.  Finally, since $q$ does not label a face of $X$, it is clear that $X_{<q} = X_{\leqslant q}$ thus concluding the proof.   

\end{proof}

Both Proposition \ref{phanScarf} and Proposition \ref{generalizationScarf} are instances of statements saying ``if a geometric object $X$ has a certain property then there is a monomial ideal whose resolution is supported on $X$''.  We believe that these propositions merely begin to give a description of what types of cell complexes can support resolutions, and that there are more statements of this type that exist with different conditions on the geometric object.

Moreover, both Proposition \ref{simplicialFilter} and Proposition \ref{GenSimpFilt} are of the form ``if a lattice $P$ has resolutions with some property then all the lattices above it with the same total Betti numbers also have the same property''.  Statements of this type are pleasantly surprising since it is known that if $Q >P$ then minimal resolutions of $Q$ will be resolutions of $P$ thus certain properties of a minimal resolution of $Q$ descend to  properties of some resolution of $P$.  Our statements however say the opposite, that certain properties of a minimal resolution of $P$ can be lifted to a some extent in $\mathcal{L}(n)$ and we believe that these are instances of a much stronger result.  This is the subject of ongoing work.


\begin{thebibliography}{GPW99}

\bibitem[Bis03]{biss}
Daniel~K. Biss.
\newblock The homotopy type of the matroid {G}rassmannian.
\newblock {\em Ann. of Math. (2)}, 158(3):929--952, 2003.

\bibitem[Bj{\"o}95]{topMeth}
A.~Bj{\"o}rner.
\newblock Topological methods.
\newblock In {\em Handbook of combinatorics, {V}ol.\ 1,\ 2}, pages 1819--1872.
  Elsevier, Amsterdam, 1995.

\bibitem[BPS98]{bps}
Dave Bayer, Irena Peeva, and Bernd Sturmfels.
\newblock Monomial resolutions.
\newblock {\em Math. Res. Lett.}, 5(1-2):31--46, 1998.

\bibitem[BS98]{bs}
Dave Bayer and Bernd Sturmfels.
\newblock Cellular resolutions of monomial modules.
\newblock {\em J. Reine Angew. Math.}, 502:123--140, 1998.

\bibitem[Fl{\o}09]{floystad}
Gunnar Fl{\o}ystad.
\newblock Cellular resolutions of {C}ohen-{M}acaulay monomial ideals.
\newblock {\em J. Commut. Algebra}, 1(1):57--89, 2009.

\bibitem[GPW99]{GPW}
Vesselin Gasharov, Irena Peeva, and Volkmar Welker.
\newblock The lcm-lattice in monomial resolutions.
\newblock {\em Math. Res. Lett.}, 6(5-6):521--532, 1999.

\bibitem[Map09]{mapesThesis}
Sonja Mapes.
\newblock {\em Finite atomic lattices and their relationship to resolutions of
  monomial ideals}.
\newblock PhD thesis, Columbia University, 2009.

\bibitem[MSY00]{msy}
Ezra Miller, Bernd Sturmfels, and Kohji Yanagawa.
\newblock Generic and cogeneric monomial ideals.
\newblock {\em J. Symbolic Comput.}, 29(4-5):691--708, 2000.
\newblock Symbolic computation in algebra, analysis, and geometry (Berkeley,
  CA, 1998).

\bibitem[Pha06]{phan}
Jeffery Phan.
\newblock {\em Properties of Monomial Ideals and their Free Resolutions}.
\newblock PhD thesis, Columbia University, 2006.

\bibitem[PV]{velascoFrames}
Irena Peeva and Mauricio Velasco.
\newblock Frames and degenerations of monomial ideals.
\newblock {\em to appear in Trans. Amer. Math. Soc.}

\bibitem[Sta97]{sta97}
Richard~P. Stanley.
\newblock {\em Enumerative combinatorics. {V}ol. 1}, volume~49 of {\em
  Cambridge Studies in Advanced Mathematics}.
\newblock Cambridge University Press, Cambridge, 1997.
\newblock With a foreword by Gian-Carlo Rota, Corrected reprint of the 1986
  original.

\bibitem[Vel08]{velascoNonCW}
Mauricio Velasco.
\newblock Minimal free resolutions that are not supported by a {CW}-complex.
\newblock {\em J. Algebra}, 319(1):102--114, 2008.

\end{thebibliography}

\end{document}